\documentclass{article}
\usepackage[utf8]{inputenc}

\usepackage[a4paper]{geometry}
\newcommand{\nat}{\ensuremath {\mathbb N} }
\newcommand{\V}{\mathbb Var}
\newcommand{\Prob}{\mathbb{P}}

\newcommand{\E}{\mathbb E}

\newcommand{\F}{\mathcal{F}}
\newcommand{\T}{\mathcal{T}}

\newcommand{\Ss}{\mathcal{S}}

\newcommand{\mb}{\mathbb}

\usepackage[]{natbib}
\setcitestyle{aysep={}}

\usepackage{graphicx}
\usepackage[colorlinks=true, allcolors=blue]{hyperref}

\usepackage{paralist}
\usepackage{amsmath}
\usepackage{amssymb}
\usepackage{amsthm}
\usepackage{algorithm}
\usepackage{algpseudocode}
\usepackage{tikz}
\usepackage[]{todonotes}   

\newtheorem{theorem}{Theorem}[section]
\newtheorem{lemma}[theorem]{Lemma}
\newtheorem{corollary}[theorem]{Corollary}

\newtheorem{observation}[theorem]{Observation}
\newtheorem{conjecture}[theorem]{Conjecture}

\author{
Bogumi\l{} Kami\'nski\footnote{SGH Warsaw School of Economics, Warsaw, Poland} 
\and
Pawe\l{}~Pra\l{}at\thanks{Department of Mathematics, Ryerson University, Toronto, ON, Canada.}
}

\title{Subtrees of a random tree}

\begin{document}

\maketitle

\begin{abstract}
Let $\T$ be a random tree taken uniformly at random from the family of labelled trees on $n$ vertices. In this note, we provide bounds for $c(n)$, the number of sub-trees of $T$ that hold asymptotically almost surely. With computer support we show that $1.41805386^n \le c(n) \le 1.41959881^n$. Moreover, there is a strong indication that, in fact, $c(n) \le 1.41806183^n$.
\end{abstract}

\section{Introduction}

In this paper, we are concerned with the problem of finding bounds for the number of sub-trees of a random tree on $n$ vertices. Clearly, the path $P_n$ and, respectively, the star $K_{1,n-1}$ have the most and the least sub-trees among all trees of order $n$. The binary trees that maximize or minimize the number of sub-trees are characterized in~\cite{Szekely2005,Szekely2007}. There is an unexpected connection between the binary trees which maximize the number of subtrees and the binary trees which minimize the Wiener index, a chemical index widely used in biochemistry; the the \emph{Wiener index} is defined as the sum of all pairwise distances between vertices~\cite{Wiener1947}. Subtrees of trees with given order and maximum vertex degree are studied in~\cite{Kirk2008}. The extremal trees coincide with the ones for the Wiener index as well. Finally, trees with given order and given degree distribution was considered in~\cite{Zhang2013}. 

\medskip

In this paper, we investigate $c(n)$, the number of subtrees of a random tree $\T$ taken uniformly at random from the family of labelled trees on $n$ vertices. The tree $\T$ is called a \emph{random tree} (or \emph{random Cayley tree}). The classical approach to the study of the properties of $\T$ was purely combinatorial, that is, via counting trees with certain properties. In this way, R\'enyi and Szekeres, using complex analysis, investigated the height of $\T$. Perhaps surprisingly, it turns out that the typical height is of order $\sqrt{n}$~\cite{Renyi1967}. Now, a useful relationship between certain characteristics of random trees and branching processes is established. In fact, recently and independently of this work,~\cite{CaiJanson} investigated the number of subtrees in a conditioned Galton–Watson tree of size $n$. They showed that $\log(c(n))$ has a Central Limit Law and that the moments of $c(n)$ are of exponential scale. In this paper, instead of exploiting this probabilistic point of view we approach the problem through combinatorial perspective which, presumably, gives stronger asymptotic bounds for $c(n)$. For more on random trees see, for example,~\cite{BookFK} or~\cite{MR3616205}.

\medskip

Our main results are presented in Section~\ref{sec:theory}. After introducing the notation we move to a lower bound that does not require computer support; see Section~\ref{sec:lower_bound_1}. The strongest lower bound, with support of a computer, is presented in Section~\ref{sec:lower_bound_2} culminating with Theorem~\ref{thm:lower-computer} which gives $c(n) \ge 1.41805^n$. The strongest upper bound can be found in Section~\ref{sec:upper_bound}; Theorem~\ref{thm:upper-computer}
implies that $c(n) \le 1.41960^n$. In the final section of the paper, Section~\ref{sec:conclustion}, we present a conjecture (that we are rather confident is true) that would determine the first 5 digits of $\sqrt[n]{c(n)}$; see Conjecture~\ref{con:upper_bound}. There is also a short discussion of the outcome of applying the general result of~\cite{Zhang2013} on the number of subtrees of a tree with a given order and degree distribution. The final subsection discusses briefly complementary simulations that we performed during this project. 






\medskip

The numerical results presented in this paper (in particular, Tables \ref{tab:tab1} and \ref{tab:tab2}) were obtained using Julia language~\cite{Julia}. The computations were performed on AWS EC2 taking in total approximately 1,000 hours of computing.

\section{Theoretical bounds}\label{sec:theory}

\subsection{Asymptotic notation}

Each time we refer to $\T$ in this paper, we consider a labelled tree on the vertex set $[n]$ taken uniformly at random from the set of all labelled trees on $n$ vertices. As typical in random graph theory, we shall consider only asymptotic properties of $\T$ as $n\rightarrow \infty$. We emphasize that the notations $o(\cdot)$ and $O(\cdot)$ refer to functions of $n$, not necessarily positive, whose growth is bounded. We use the notations $f \ll g$ for $f=o(g)$ and $f \gg g$ for $g=o(f)$. We also write $f(n) \sim g(n)$ if $f(n)/g(n) \to 1$ as $n \to \infty$ (that is, when $f(n) = (1+o(1)) g(n)$). We say that an event in a probability space holds \emph{asymptotically almost surely} (\emph{a.a.s.}) if its probability tends to one as $n$ goes to infinity.

\subsection{Pr\"ufer code}

Let us start with recalling a classic result that will be useful in our analysis. \emph{Pr\"ufer code} of a labelled tree $T$ on $n$ vertices is a unique sequence from $[n]^{n-2}$ (the set of sequences of length $n-2$, each term is from the set $[n]=\{1,2,\ldots,n\}$) associated with tree $T$~\cite{Prufer1918}. In fact, there exists a bijection from the family of labelled trees on $n$ vertices and the set $[n]^{n-2}$. This, in particular,  implies that the Cayley's formula holds: the number of labelled trees on $n$ vertices is $n^{n-2}$. More importantly, it gives us a way to generate a random labelled tree by simply selecting a random element from $[n]^{n-2}$ and considering the corresponding tree $\T$. 

Suppose a labelled tree $T$ has a vertex set $[n]$. One can generate a Pr\"ufer code of $T$ by iteratively removing vertices from the tree until only two vertices remain. At step $i$ of this process, remove the leaf with the smallest label and set the $i$th element of the Pr\"ufer code to be the label of this leaf's neighbour.

\subsection{Lower bound: trivial approach}

Consider the Pr\"ufer code of $\T$. Clearly, the degree of any vertex $v$ is the number of times $v$ appears in the code plus 1. It follows that for any $v \in [n] = V(\T)$ and any $k\in \nat$,
\begin{equation}\label{eq:degree_distr}
\Prob \left( \deg(v) = k \right) = \binom{n-2}{k-1} \left( \frac {1}{n} \right)^{k-1} \left( 1 - \frac {1}{n} \right)^{n-k-1}  \sim \frac {e^{-1}}{(k-1)!}. 
\end{equation}

Now, let $X_1$ be the number of leaves of $\T$. From above it follows that $\E[X_1] \sim n/e$ and we can easily prove (using, say, the second moment method) that a.a.s.\ $X_1 \sim n / e$. (We will prove a more general result below---see Lemma~\ref{lemma:types}---so we skip a formal argument here.) One can select \emph{all} non-leaves and then \emph{any} subset of the leaves to form a sub-tree. (Note that any subset of leaves can be safely removed and so any choice results with a connected graph.) We get the following lower bound that holds a.a.s.:
$$ 
c(n) \ge 2^{X_1} = 2^{(1/e + o(1))n} = \left( 2^{1/e} + o(1) \right)^n \ge 1.29045^n.
$$ 

\subsection{Lower bound: warming up on a piece of paper...}\label{sec:lower_bound_1}

The reason for this section is twofold. First of all, we present a lower bound that does not require computer support. Another reason is to prepare the reader for a more sophisticated argument presented in the next section that will give a stronger bound but will require computer support. 

\begin{theorem}\label{thm:lower-trivial}
A.a.s.\ $c(n) \ge 1.37135^n$.
\end{theorem}
\begin{proof}
Let $\gamma$ be a sufficiently large integer that will be determined soon. For $k \in \{2, 3, \ldots, \gamma\}$, let $X_k$ be the number of subsets $S \subseteq [n]$ of size $k$ that induce a star ($K_{1,k-1}$) and the only edge connecting $S$ to the rest of $\T$ is adjacent to the center of the star. In particular, the $k-1$ leaves of the star are leaves in $\T$.

Trivial, but important, property is that vertices of $\T$ that belong to $K_{1,k-1}$ cannot be part of some other $K_{1,k'-1}$ for some $k'$ (that could be equal to $k$ but does not have to be). We put vertex $v$ of $\T$ (together with the $k-1$ leaves adjacent to $v$) into $C_k$ if $v$ belongs to some $K_{1,k-1}$. As a result, we partition the vertex set into a family of classes $C_k$ ($k\in\{2, 3, \ldots, \gamma\}$; $C_k$ contains $X_k$ stars and so it contains $X_k \cdot k$ vertices), leaves $L$ that are not part of any earlier class, and $R$ that contains the remaining vertices of $\T$.  

By considering a random Pr\"ufer code, we get that a.a.s., for any $k\in\{2, 3, \ldots, \gamma\}$
$$
X_k \sim \binom{n}{k} k \left( \frac {1}{n} \right)^{k-1} \left(1- \frac {k}{n} \right)^{n-k-1} \sim \frac {n e^{-k}}{(k-1)!};
$$
there are $\binom{n}{k}$ choices for $S$, $k$ choices for the root, each leaf selects the root with probability $1/n$, with probability $(1-k/n)^{n-k-1}$ no vertex picked leaves and no vertex other than the leaves picked the root. (More general result will be proved in the next subsection---see Lemma~\ref{lemma:types}.) The number of leaves in $L$ is a.a.s.
$$
|L| ~~=~~ |X_1|-\sum_{k=2}^{\gamma} X_k \cdot (k-1) ~~\sim~~ \left( e^{-1} - \sum_{k=2}^{\gamma} \frac {e^{-k}}{(k-1)!} (k-1) \right) n ~~=~~ \beta_L\ n,
$$
where $\beta_L=\beta_L(\gamma)$ is the constant that can be made arbitrarily close to 
$$
\hat{\beta}_L ~~:=~~ e^{-1} - \sum_{k\ge 2} \frac {e^{-k}}{(k-2)!} ~~=~~ e^{-1} - e^{1/e-2} ~~\approx~~ 0.1724
$$
by taking $\gamma$ large enough. 
The number of rooted sub-trees of $K_{1,{k-1}}$ (including the empty tree) is clearly $2^{k-1}+1$. Hence, we get the following lower bound for $c(n)$ by taking all vertices of $R$, any subset of $L$, and any rooted sub-trees from classes $C_k$: a.a.s.
\begin{eqnarray*}
c(n) &\ge& 2^{ |L| } \prod_{k=2}^{\gamma} \left( 2^{k-1} + 1 \right)^{X_k} ~~=~~ \left( 2^{\beta_L+o(1)} \prod_{k=2}^{\gamma} (2^{k-1}+1)^{e^{-k}/(k-1)!+o(1)} \right)^n \\
&=& \left( 2^{\beta_L} \prod_{k=2}^{\gamma} (2^{k-1}+1)^{e^{-k}/(k-1)!} +o(1) \right)^n = \Big( \beta + o(1) \Big)^n,
\end{eqnarray*}
where $\beta=\beta(\gamma)$ is the constant that can be made arbitrarily close to 
$$
\hat{\beta} ~~:=~~  2^{\hat{\beta}_L} \prod_{k\ge2} (2^{k-1}+1)^{e^{-k}/(k-1)!} = 2^{e^{-1} - e^{1/e-2}} \prod_{k\ge2} (2^{k-1}+1)^{e^{-k}/(k-1)!} > 1.37135
$$
by taking $\gamma$ large enough. The desired bound holds.
\end{proof}

\subsection{Lower bound: computer assisted argument}\label{sec:lower_bound_2}

In this section, we generalize the strategy we considered in the previous section. Instead of restricting ourselves to stars, we investigate all possible trees on $k$ vertices, where $k \le K$ for some value of $K$. Unfortunately, it seems impossible to find a closed formula for the number of trees with a given number of sub-trees but, with computer assist, we can do it even for relatively large values of $K$. As before, one could include an (arbitrarily large) family of stars but this improvement is negligible and so we do not do it. 

\medskip

Fix some $K \in \nat$. We start with a few important definitions.

\subsubsection*{Family $\F_k$}

For each $k \in [K]$, let $\F_k$ be the family of rooted trees on $k$ vertices; that is, each member of $\F_k$ is a pair $(T,v)$, where $T$ is a labelled tree on the vertex set $[k]$ and $v \in [k]$. Clearly, $|\F_k| = k^{k-2} \cdot k = k^{k-1}$. Finally, let $\F = \bigcup_{k=1}^K \F_k$.

\subsubsection*{Vertices of type $(T,v)$ and internal vertices}

For each vertex $v$ of $\T$, we consider $\ell=\deg(v)$ sub-trees of $\T$ ($T_1, T_2, \ldots, T_\ell$), all of them rooted at $v$, that are obtained by removing one of the $\ell$ edges adjacent to $v$. Now, each $T_i$ (on $k_i$ vertices) is re-labelled so that labels are from $[k_i]$ but the relative order is preserved. Since we aim for asymptotic results, we may assume that $n > 2K$ and so at most one such rooted tree, say $(T_1,v)$, belongs to $\F$. If this is the case, then we say that $v$ is \emph{of type $(T_1,v)$} and that it \emph{induces} rooted tree $(T_1,v)$; otherwise, we say that $v$ is an \emph{internal} vertex.

\subsubsection*{Partition of the vertex set of $\T$}

We partition the vertex set of $\T$ (set $[n]$) as follows. We start the process at \emph{round} $K$. (It will be convenient to count rounds from $K$ down to 1.) For each vertex of type $(T,v)$, for some $(T,v) \in \F_K$, we put all the vertices of the rooted sub-tree it induces into class $C(T,v)$. Note that no vertex of $\T$ belongs to more than one sub-tree as we consider only types from $\F_K$ (trees of a fixed size). Hence, in particular, the classes created so far are mutually disjoint. On the other hand, all vertices of type different than $(T,v)$ that are placed into class $C(T,v)$ are of type from $\F \setminus \F_K$. Hence, in order to avoid placing one vertex into more than one class, we need to ``trim'' the tree and remove all vertices that are already placed into some class. Round $K$ is finished and now we move to the next round, round $K-1$, in which vertices of types from $\F_{K-1}$ are considered and proceed the same way. (Note that not all of them are removed during round $K$.) We do it recursively all the way down to round $1$ during which $\F_1$ is considered and so the remaining leaves of $\T$ are trimmed. The only vertices left are internal one which are placed into set $R$. We obtain the following partition of $[n]$:
$
\{ C(T,v) : (T,v) \in \F \} \cup \{ R \}.
$

\bigskip

We start with estimating the number of vertices of each type.

\begin{lemma}\label{lemma:types}
For any $K \in \nat$, the following property holds a.a.s. For any $(T,v) \in \F_k$ for some $k \in [K]$, the number of vertices of type $(T,v)$ is $(1+o(1)) n e^{-k}/k!$.
\end{lemma}
\begin{proof}
The argument is a straightforward application of the second moment method. Fix any $k \in [K]$ and $(T,v) \in \F_k$; we will show that the desired bound holds a.a.s.\ for this choice. This will finish the proof as the number of choices for $k$ and $(T,v)$ is bounded and so the conclusion holds by the union bound.

For any $S \subseteq [n]$, $|S|=k$, let $I(S)$ be the indicator random variable that set $S$ induces a tree $T$ rooted at $v$ (after relabelling preserving the order of vertices of $S$) and the only edge from $S$ to its complement is adjacent to a vertex re-labelled as $v$. The number of vertices of type $(T,v)$ is
$$
X = \sum_{S \subseteq [n], |S|=k} I(S).
$$
For any $S$ we have 
\begin{eqnarray*}
p &:=& \Prob(I(S) = 1) = \left( \frac {1}{n} \right)^{k-1} \left( 1 - \frac {k}{n} \right)^{n-k-1} \sim n^{-(k-1)} e^{-k}.
\end{eqnarray*}
Indeed, without loss of generality, we may assume that $S = \{1, 2, \ldots, k\}$. Then, the first $k-1$ terms of the Pr\"ufer code of $\T$ are completely determined by $T$ and $v$ 	(hence term $(1/n)^{k-1}$); moreover, the remaining $(n-2)-(k-1)=n-k-1$ terms cannot be from $S$ (hence term $(1-k/n)^{n-k-1}$). It follows that
$$
\E [X] = \binom{n}{k} p \sim \frac {n^k p}{k!} \sim \frac {ne^{-k}}{k!}.
$$

Now, 
\begin{eqnarray*}
\V [x] &=& \V \left[ \sum_{S \subseteq [n], |S|=k} I(S) \right] \\
&=& \sum_{S,S' (*)} \left( \Prob(I(S)=1, I(S')=1) - \Prob(I(S)=1)^2 \right) \\
&& \quad + \sum_{S} \left( \Prob(I(S)=1) - \Prob(I(S)=1)^2 \right),
\end{eqnarray*}
where $(*)$ means that the sum is taken over all pairs of sets $S, S' \subseteq [n]$ with $|S|=|S'|=k$. The second term in the last sum can be dropped to get and upper bound of $\E[X]$ for the last sum. More importantly, note that if $S$ and $S'$ intersect, then $\Prob(I(S)=1, I(S')=1)=0$. Hence,
$$
\V [x] ~~\le~~ \sum_{S,S' (**)} \left( \Prob(I(S)=1, I(S')=1) - \Prob(I(S)=1)^2 \right) + \E[X],
$$
where $(**)$ means that the sum is taken over all pairs of disjoint sets $S, S' \subseteq [n]$ with $|S|=|S'|=k$. For any such pair,
\begin{eqnarray*}
q &:=& \Prob(I(S)=1, I(S')=1) - \Prob(I(S)=1)^2 \\
&=& \left( \frac {1}{n} \right)^{2(k-1)} \left( 1 - \frac {2k}{n} \right)^{(n-2)-2(k-1)} - \left( \left( \frac {1}{n} \right)^{k-1} \left( 1 - \frac {k}{n} \right)^{(n-2)-(k-1)} \right)^2 \\
&=& \left( \frac {1}{n} \right)^{2(k-1)} \left( \left( 1 - \frac {2k}{n} \right)^{n-2k} - \left( 1 - \frac {k}{n} \right)^{2n-2k-2} \right).
\end{eqnarray*}
Using the fact that $1-x = \exp (-x-x^2/2 + O(x^3))$ and then that $\exp(x) = 1 + x + O(x^2)$, we get
\begin{eqnarray*}
q &:=& n^{-2(k-1)} \left( \exp \left( -2k + \frac {2k^2}{n} + O(n^{-2}) \right) - \exp \left( -2k + \frac {k^2 + 2k}{n} +O(n^{-2}) \right) \right) \\
&\sim& n^{-2(k-1)} e^{-2k} \left( 1 - \exp \left( \frac {-k^2 + 2k}{n} +O(n^{-2}) \right) \right) ~~\sim~~ \frac {p^2 (k^2-2k)}{n}.
\end{eqnarray*}
It follows that
$$
\V [x] ~~\le~~ \binom{n}{k} \binom{n-k}{k} q + \E[X] ~\sim~ \left( \binom{n}{k} p \right)^2 \frac {k^2-2k}{n} + \E[X] ~=~ o(\E[X]^2).
$$
The second moment method implies that a.a.s.\ $X \sim \E[X]$ and the proof is finished.
\end{proof}

\bigskip

Now, we are ready to analyze the trimming process that yields the desired partition of the vertex set of $\T$. 

\begin{lemma}\label{lemma:c_tv}
For any $K \in \nat$, the following property holds a.a.s. For any $(T,v) \in \F_k$ for some $k \in [K]$,
$$
\frac {|C(T,v)|}{n} \sim k \cdot f_K(k), \quad \text{ where } f_K(k) := \frac {e^{-k}}{k!} - \sum_{\ell=k+1}^{K} 
(\ell-k)^{\ell-k-1} \binom{\ell}{\ell-k} \frac {e^{-\ell}}{\ell!}.
$$
\end{lemma}

\begin{proof}
Since we aim for a statement that holds a.a.s., we may assume that $\T$ is any labelled tree on the vertex set $[n]$ that satisfies properties stated in Lemma~\ref{lemma:types}. The desired property will hold deterministically. To that end, we need to analyze the trimming process. 

Fix any $(T,v) \in \F_K$. During the first round (that is, round $K$), all vertices of type $(T,v)$, together with the corresponding trees that are induced by them, are moved to class $C(T,v)$. By Lemma~\ref{lemma:types}, the number of vertices of type $(T,v)$ is $(1+o(1)) ne^{-K}/K!$ and so $|C(T,v)|/n \sim K \cdot \hat{f}_K(K)$, where
$$
\hat{f}_K(K) := \frac {e^{-K}}{K!}.
$$
Now, consider any round $k$ ($1 \le k < K$) and suppose that the process is already analyzed up to that point; that is, during rounds $\ell$ ($k+1 \le \ell \le K$), for any $(T,v) \in \F_\ell$, $(1+o(1)) \hat{f}_K(\ell) n$ vertices of type $(T,v)$ were moved to class $(T,v)$ (as usual, together with the corresponding trees that are induced by them). Fix any $(T,v) \in \F_k$. By Lemma~\ref{lemma:types}, at the beginning of the trimming process there were $(1+o(1)) ne^{-k}/k!$ vertices of type $(T,v)$. Some of them were trimmed during some round $\ell$ ($k+1 \le \ell \le K$); but how many of them? In order to answer this question we need to know how many rooted trees on $\ell$ vertices contain a vertex of type $(T,v)$. We are going to use an argument similar to the one used in the proof of Lemma~\ref{lemma:types}. There are $\binom{\ell}{k}$ ways to select labels for the sub-tree on $k$ vertices of a tree on $\ell$ vertices. Without loss of generality, we may assume that the selected labels are $\{1,2,\ldots, k\}$. Now, the Pr\"ufer code for a super-tree on $\ell$ vertices has to have the first $k-1$ terms as determined by $T$ and $v$. The remaining $(\ell-2)-(k-1)=\ell-k-1$ terms yield all possible super-trees; each of these terms is from $[\ell]\setminus [k]$. Since there are $(\ell-k)$ choices for the root of a tree on $\ell$ vertices, we get that the answer to our question is $(\ell-k)^{\ell-k-1} (\ell-k) \binom{\ell}{k} = (\ell-k)^{\ell-k}\binom{\ell}{k}$. It follows that the number of vertices of type $(T,v)$ that survived till round $k$ is $(1+o(1)) \hat{f}_K(k) n$, where
\begin{equation}\label{eq:gk}
\hat{f}_K(k) := \frac {e^{-k}}{k!} - \sum_{\ell=k+1}^K (\ell-k)^{\ell-k} \binom{\ell}{\ell-k} \hat{f}_K(\ell),
\end{equation}
and so $|C(T,v)|/n \sim k \cdot \hat{f}_K(k)$. 

It remains to show that $f_K(k) = \hat{f}_K(k)$ for $1 \le k \le K$; we prove it by strong induction on $k$. Clearly, $f_K(K)=\hat{f}_K(K)$ so the base case holds. Suppose then that $f_K(\ell) = \hat{f}_K(\ell)$ for $k+1 \le \ell \le K$ and our goal is to show that $f_K(k)=\hat{f}_K(k)$. From this and~(\ref{eq:gk}) we get 
\begin{eqnarray*}
\hat{f}_K(k) &=& \frac {e^{-k}}{k!} - \sum_{\ell=k+1}^K (\ell-k)^{\ell-k} \binom{\ell}{\ell-k} f_K(\ell) \\
&=& \frac {e^{-k}}{k!} - \sum_{\ell=k+1}^K (\ell-k)^{\ell-k} \binom{\ell}{\ell-k}
\left(
    \frac {e^{-\ell}}{\ell!} - \sum_{m=\ell+1}^{K}
    (m-\ell)^{m-\ell-1} \binom{m}{m-\ell} \frac {e^{-m}}{m!}
\right).
\end{eqnarray*}
We will show that the terms in $\hat{f}_K(k)$ containing $e^{-a}$ for $k<a\leq K$ are the same as the ones in $f_K(k)$. (Clearly, it is the case for $a=k$.) To see this, note that one of these terms is present in the above equation for $\ell=a$ (see the first part inside the parenthesis) and one for each $k<\ell<a$ (see the term corresponding to $m=a$ in the second part inside the parenthesis). Collecting those terms in $\hat{f}_K(k)$ we get:
$$
-(a-k)^{a-k} \binom{a}{a-k} \frac{e^{-a}}{a!} +
\sum_{\ell=k+1}^{a-1}
(\ell-k)^{\ell-k} \binom{\ell}{\ell-k}
(a-\ell)^{a-\ell-1} \binom{a}{a-\ell} \frac {e^{-a}}{a!}.
$$
On the other hand, the only term in $f_K(k)$ containing $e^{-a}$ is $-(a-k)^{a-k-1} \binom{a}{a-k} e^{-a} / a!$. Hence, to finish the inductive step it is enough to show that
$$
(a-k-1)(a-k)^{a-k-1} \binom{a}{a-k} = 
\sum_{\ell=k+1}^{a-1}
(\ell-k)^{\ell-k} \binom{\ell}{\ell-k}
(a-\ell)^{a-\ell-1} \binom{a}{a-\ell},
$$
which, after substituting $b=a-k$ and $c=\ell-k$, we can rewrite as
\begin{equation}\label{eq:star}
(b-1)b^{b-1}=
\sum_{c=1}^{b-1} \binom{b}{c} c^c (b-c)^{b-c-1}=
\sum_{c=1}^{b-1}c \binom{b}{c}
c^{c-1} (b-c)^{b-c-1}.
\end{equation}
Then, by setting $d=b-c$ in the first step, and then using the fact that $\binom{b}{b-d}= \binom{b}{d}$ and changing $d$ to $c$ in the notation in the second step, we get
\begin{eqnarray}
(b-1)b^{b-1} &=&
\sum_{d=1}^{b-1}(b-d) \binom{b}{b-d}
(b-d)^{b-d-1} d^{d-1} \nonumber \\
&=& \sum_{c=1}^{b-1}(b-c) \binom{b}{c}
(b-c)^{b-c-1}c^{c-1}. \label{eq:starr}
\end{eqnarray}
By adding~(\ref{eq:star}) and~(\ref{eq:starr}) and dividing both sides by $b$ we get
\begin{equation}\label{eq:final_piece}
2(b-1)b^{b-2}=
\sum_{c=1}^{b-1} \binom{b}{c}
c^{c-1}(b-c)^{b-c-1}.
\end{equation}

The final ``puzzle piece'' missing is the proof of~(\ref{eq:final_piece}) for which we will use a bijective argument. The left hand side of~(\ref{eq:final_piece}) counts all labelled trees on the vertex set $[b]$ with one edge selected and \emph{oriented}. 
Now, consider the following construction. First, take any proper and non-empty subset $C \subseteq [b]$ of size $c$ ($1 \le c \le b-1$); let $D = [b] \setminus C$. Construct any labelled tree on $C$ and select one vertex $v_C \in C$. Similarly, construct any labelled tree on $D$ and select one vertex $v_D \in D$. Finally, connect $v_C$ to $v_D$ by an oriented edge from $v_C$ to $v_D$.
The described construction generates all possible labelled trees with one edge selected and oriented. Moreover, each such tree is constructed exactly once. Now observe that number of such constructions is equal to right hand side of~(\ref{eq:final_piece}). The proof is finished.
\end{proof}

\medskip

Now, we are ready to state the main result of this sub-section that yields the strongest lower bound we have.

\begin{theorem}\label{thm:lower-comp}
Fix any $K \in \nat$. 
For any $(T,v) \in \F_k$ for some $k \in [K]$, let $g(T,v)$ be the number of sub-trees of $T$ containing $v$. 
Let $f_K(k)$ be defined as in the statement of Lemma~\ref{lemma:c_tv}. 

Then, the following bound holds a.a.s. 
\begin{eqnarray}
c(n) &\ge& \left( \prod_{k=1}^K \prod_{(T,v) \in \F_k} g(T,v)^{f_K(k)} + o(1) \right)^n.\label{eq:cn}
\end{eqnarray}
\end{theorem}
\begin{proof}
Recall that the vertex set of $\T$ is partitioned as follows: for $(T,v) \in \F$, set $C(T,v)$ contains vertices of type $(T,v)$ that induce rooted trees $T$, together with other vertices of $T$; the internal vertices form set $R$. It follows from Lemma~\ref{lemma:c_tv} that a.a.s., for any $k \in [K]$ and any $(T,v) \in \F_k$, the number of rooted trees in $C(T,v)$ is $(1+o(1)) f_K(k) n$. 
By taking all vertices of $R$ and any rooted sub-trees from $C(T,v)$, the following lower bound for $c(n)$ holds: a.a.s.
\begin{equation*}
c(n) \ge \left( \prod_{k=1}^K \prod_{(T,v) \in \F_k} g(T,v)^{f_K(k)+o(1)} \right)^n = \left( \prod_{k=1}^K \prod_{(T,v) \in \F_k} g(T,v)^{f_K(k)} + o(1) \right)^n,
\end{equation*}
since the number of terms in this product is bounded. 
\end{proof}

Function $f_K(k)$ can be easily calculated (numerically) even for relatively large values of $K$ and $k$. Unfortunately, there is no closed formula for $g(T,v)$, the number of rooted sub-trees of $T$ (recall that the empty tree is included). On the other hand, $g(T,v)$ can be easily computed with computer support using the following simple, recursive algorithm. Let $N(v)$ be the set of neighbours of $v$. For any $w \in N(v)$, $T-vw$ (that is, forest obtained after removing edge $vw$) consists of two sub-trees; let $S(T,v,w)$ be the sub-tree containing $w$. Then $g(T,v)$ can be computed as follows: if $T$ is $K_1$ (isolated vertex), then $g(T,v)=2$; otherwise,
\begin{equation}
g(T,v)=1+\prod_{w\in N(v)} g(S(T, v, w), w).\label{eq:rec}
\end{equation}
Actual computations of $c(n)$ can be made efficient using the following two observations:
\begin{enumerate}
    \item we do not have to explicitly generate all trees $(T,v)$ in $\F_k$; it is enough to count the number of rooted trees of size $k$ that have a given value of $g(T,v)$---since this is enough to compute~\eqref{eq:cn};
    \item if we start from $k=1$ up to $k=K$, then we can derive counts of trees from $\F_k$ with unique values of $g(T,v)$ using counts of numbers of trees from $\F_{k-s}$, where $s\in[k-1]$, with unique values of $g(T,v)$---as in~\eqref{eq:rec}, the right hand side considers trees of size one less than the left hand side.
\end{enumerate}
The exact procedure is given in Algorithm \ref{alg:sub}, where
$x(k,g)= \left|\left\{(T,v)\in\F_k:g(T,v)=g\right\}\right|$. Using $x(k,g)$, one can rewrite~\eqref{eq:cn} as follows:
\[
c(n) \ge \left( \prod_{k=1}^K \left(\prod_{g\in\nat} g^{f_K(k)}\right)^{x(k,g)} + o(1) \right)^n
= \left( \prod_{k=1}^K \left(\prod_{g\in\nat} g^{x(k,g)}\right)^{f_K(k)} + o(1) \right)^n,
\]

\begin{algorithm}
    \caption{Algorithm for calculation of $x(k,g)$. \label{alg:sub}}
    \begin{algorithmic}
        \State $\forall k,g\in\nat: x(k, g) \gets 0$
        \State $x(1, 2) \gets 1$
        \For{$k\in \{2,3,\dots,K\}$}
            \ForAll{$a_1, a_2, \dots, a_m \in \nat$ such that $\sum_{i=1}^ma_i=k-1$ and $a_i\le a_{i+1}$}
                \State let $n_j$, $j\in[p]$, be the length of the $j$-th constant subsequence of the $a_i$ sequence
                \ForAll{$x(a_i, g_i)$ over all $i\in[m]$ and $g_i\in\nat$}
                    \State $x(k, 1+\prod_{i=1}^{m}g_i) \gets x(k, 1+\prod_{i=1}^{m}g_i) +
                    \frac{k!}{\prod_{j=1}^pn_j!}
                    \prod_{i=1}^{m}\frac{x(a_i,g_i)}{a_i!}$
                \EndFor
            \EndFor
        \EndFor
    \end{algorithmic}
\end{algorithm}

The obtained lower bounds for $K=1, 2, \ldots, 30$ are presented in Table~\ref{tab:tab1} (column \emph{lower}, the following columns are explained in the following sections). Clearly, the strongest bound is yielded by $K=30$ and is the best lower bound we have.

\begin{theorem}\label{thm:lower-computer}
A.a.s.\ $c(n) \ge 1.41805^n$.
\end{theorem}

\subsection{Upper bound: trivial approach}

Recall that in the proof of Theorem~\ref{thm:lower-trivial} we partition the vertex set of $\T$ into a family of classes $C_k$ ($k\in\{2, 3, \ldots, \gamma\}$; $C_k$ contains $X_k$ stars and so it contains $X_k \cdot k$ vertices), leaves $L$ that are not part of any earlier class, and $R$ that contains the remaining vertices of $\T$. The size of $L$ is already estimated in Theorem~\ref{thm:lower-trivial}. The number of vertices that belong to some class $C_k$ is a.a.s.
$$
\Big| \bigcup_{k=2}^{\gamma} C_k \Big| ~~=~~ \sum_{k=2}^{\gamma} X_k \cdot k ~~\sim~~ \sum_{k=2}^{\gamma} \frac {n e^{-k}}{(k-1)!} \ k ~~=~~ \beta_C\ n,
$$
where $\beta_C=\beta_C(\gamma)$ is the constant that can be made arbitrarily close to 
$$
\hat{\beta}_C ~~:=~~ \sum_{k\ge 2} \frac {e^{-k}}{(k-1)!} \ k ~~=~~ e^{1/e-1} + e^{1/e-2} - e^{-1} ~~\approx~~ 0.3591 \ n
$$
by taking $\gamma$ large enough. 
Finally, a.a.s.
$$
|R| ~~=~~ n - \Big| \bigcup C_k \Big| - |L| ~~\sim~~ \left( 1 - \beta_C - \beta_L \right) n ~~=~~ \beta_R \ n,
$$
where $\beta_R=\beta_R(\gamma)=1-\beta_C-\beta_L$ is the constant tending to 
$$
\hat{\beta}_R ~~:=~~ 1 - \hat{\beta}_C - \hat{\beta}_L ~~=~~ 1 - e^{1/e-1} ~~\approx~~ 0.4685 \ n
$$
as $\gamma \to \infty$. 

To get an upper bound for $c(n)$, we select any subset of $R \cup L$ and any rooted sub-trees from classes $C_k$. Clearly, each sub-tree of $\T$ is achieved but not all selected sets induce a connected graph. (In fact, almost all of them do not!) So we are clearly over-counting but the following can serve as the upper bound that holds a.a.s.:
\begin{eqnarray*}
c(n) &\le& 2^{ |L|+|R| } \prod_{k=2}^{\gamma} \left( 2^{k-1} + 1 \right)^{X_k} ~~=~~ \left( 2^{\beta_L+\beta_R+o(1)} \prod_{k=2}^{\gamma} (2^{k-1}+1)^{e^{-k}/(k-1)!+o(1)} \right)^n \\
&=& \left( 2^{\beta_L+\beta_R} \prod_{k=2}^{\gamma} (2^{k-1}+1)^{e^{-k}/(k-1)!} +o(1) \right)^n = \Big( \alpha + o(1) \Big)^n,
\end{eqnarray*}
where $\alpha=\alpha(\gamma)$ is the constant that can be made arbitrarily close to 
\begin{eqnarray*}
\hat{\alpha} &:=&  2^{\hat{\beta}_L+\hat{\beta}_R} \prod_{k\ge2} (2^{k-1}+1)^{e^{-k}/(k-1)!} \\
&=& 2^{1+e^{-1} -e^{1/e-1}-e^{1/e-2}} \prod_{k\ge2} (2^{k-1}+1)^{e^{-k}/(k-1)!} < 1.89756
\end{eqnarray*}
by taking $\gamma$ large enough. It follows that a.a.s.\ $c(n) \le 1.89756^n$.

\bigskip

The same trivial argument can be used to adjust Theorem~\ref{thm:lower-comp}: the ratio between the upper and the lower bound is $2^{|R|}$, where $R$ is the set of internal vertices. The following straightforward corollary of Lemma~\ref{lemma:types} estimates the size of $R$. It shows that the fraction of vertices that are internal is tending to zero as $K \to \infty$. This is, of course, a desired property as it implies that the gap between the upper and the lower bound for $c(n)$ can be made arbitrarily small by considering large values of $K$. Unfortunately, the rate of convergence is not so fast.

\begin{corollary}\label{cor:R}
For any $K \in \nat$, a.a.s. 
$$
\frac {|R|}{n} \sim h(K) := 1 - \sum_{k=1}^K \frac {(k/e)^{k}}{k\cdot k!} = \Theta \left( \frac {1}{\sqrt{K}} \right),
$$
where an asymptotic is with respect to $K$. 
\end{corollary}
\begin{proof}
The number of internal vertices (that is, vertices that are not of type $(T,v)$ for any $(T,v) \in \F$) can be estimated using Lemma~\ref{lemma:types}. Since $|\F_k| = k^{k-1}$, we get that
$$
\frac {|R|}{n} ~~\sim~~ 1 - \sum_{k=1}^K |\F_k| \frac {e^{-k}}{k!} ~~=~~ 1 - \sum_{k=1}^K \frac {(k/e)^{k}}{k\cdot k!} ~~=~~ h(K).   
$$

To see the second part, consider the branching process in which every individual produces individuals that is an independent random variable with Poisson distribution with expectation 1. The process extincts with precisely $k$ individuals (in total) with probability $\frac {(k/e)^{k}}{k\cdot k!}$ (see, for example,~\cite{Tanner61}). Hence, $h(K)$ is the probability that the total number of individuals is more than $K$. Since the process extincts with probability 1, $h(K) \to 0^+$ as $K \to \infty$; or, alternatively, $\sum_{k \ge 1} \frac {(k/e)^{k}}{k\cdot k!} = 1$. To see the rate of convergence we apply Stirling's formula $k! \sim \sqrt{2\pi k} (k/e)^k$ to get
$$
h(K) ~~=~~ \sum_{k>K} \frac {(k/e)^{k}}{k\cdot k!} ~~=~~ \Theta \left( \sum_{k>K} k^{-3/2} \right) ~~=~~ \Theta \left( \frac {1}{\sqrt{K}} \right).
$$
The proof is finished.
\end{proof}

We get the following counterpart of Theorem~\ref{thm:lower-comp}. 

\begin{observation}
Fix any $K \in \nat$. 
For any $(T,v) \in \F_k$ for some $k \in [K]$, let $g(T,v)$ be the number of sub-trees of $T$ containing $v$. 
Let $f_K(k)$ and $h(K)$ be defined as in the statements of Lemma~\ref{lemma:c_tv} and Corollary~\ref{cor:R}, respectively. 

Then, the following bound holds a.a.s. 
\begin{eqnarray*}
c(n) &\le& \left( 2^{h(K)} \prod_{k=1}^K \prod_{(T,v) \in \F_k} g(T,v)^{f_K(k)} + o(1) \right)^n.
\end{eqnarray*}
\end{observation}

The numerical values of the upper bounds for $c(n)$ and for $|R|/n$ ($K \in \{1, 2, \ldots, 30\}$) are presented in Table~\ref{tab:tab1} (see column \emph{upper 1} and column $|R|/n$, respectively). For $K=30$ we get that a.a.s.\ $c(n) \le 1.56727^n$. As already mentioned, unfortunately, the rate of convergence is not so fast. Since the computational complexity of the problem makes $K$ to be not so large (at most $30$), the number of internal vertices is substantial ($|R|\approx 0.14434 n$ for $K=30$) and so more sophisticated arguments will be needed. 

\subsection{Upper Bound: computer assisted argument}\label{sec:upper_bound}

We continue using the notation and definitions used in Section~\ref{sec:lower_bound_2}. Recall that the vertex set of $\T$ is partitioned there as follows: for $(T,v) \in \F$, set $C(T,v)$ contains vertices of type $(T,v)$ that induce rooted trees $T$, together with other vertices of $T$; the internal vertices form set $R$. However, this time we additionally partition $R$ into two sets: $R_L$ contains vertices of type $(T,v) \in \F_k$ for some $K < k \le \hat{K}$ (\emph{light} internal vertices) and $R_H = R \setminus R_L$ (\emph{heavy} internal vertices).

\medskip

Here is the strongest upper bound we have, in its general form.

\begin{theorem}\label{thm:upper-comp}
Fix any $K, \hat{K} \in \nat$ such that $K < \hat{K}$. 
For any $(T,v) \in \F_k$ for some $k \in [K]$, let $g(T,v)$ be the number of sub-trees of $T$ containing $v$. 
Let $f_K(k)$ and $h(K)$ be defined as in the statements of Lemma~\ref{lemma:c_tv} and Corollary~\ref{cor:R}, respectively. 

Then, the following bound holds a.a.s. 
$$
c(n) \le \left( \xi_1 \xi_2 \xi_3 \xi_4 + o(1) \right)^n,
$$
where
\begin{eqnarray*}
\xi_1 &=& \left( \frac {\hat{K}+1}{\hat{K}} \right)^{h(\hat{K})} \\
\xi_2 &=& \prod_{k=K+1}^{\hat{K}} \left( \frac{k+1}{k} \right)^{k(k-1)^{k-2} e^{-k}/k!} \\
\xi_3 &=& \prod_{k=K+1}^{\hat{K}} \left( \frac{2k-1}{2k-2} \right)^{ ( k^{k-1}-k(k-1)^{k-2} ) e^{-k}/k!} \\
\xi_4 &=& \prod_{k=1}^K \prod_{(T,v) \in \F_k} g(T,v)^{f_K(k)}.
\end{eqnarray*}
\end{theorem}

\begin{proof}
Let us fix any vertex $r \in [n]$.  Our goal is to use~(\ref{eq:rec}) to estimate $g(\T,r)$, the number of sub-trees of $\T$ containing $r$. As mentioned earlier, $[n]$ is partitioned into sets $C(T,v)$ containing trees rooted at vertices of type $(T,v)$, $R_L$ and $R_H$ consisting of light and, respectively, heavy internal vertices. It follows from Lemma~\ref{lemma:c_tv} that a.a.s., for any $k \in [K]$ and any $(T,v) \in \F_k$, the number of rooted trees in $C(T,v)$ is $(1+o(1)) f_K(k) n$. From Corollary~\ref{cor:R} we get that a.a.s.\ the number of heavy internal vertices is $(1+o(1)) h(\hat{K}) n$. Finally, Lemma~\ref{lemma:types} implies that a.a.s.\ for any $(T,v) \in \F_k$ for some $k \in [\hat{K}]$, the number of vertices of type $(T,v)$ is $(1+o(1)) n e^{-k}/k!$. 

Recall that for any $w \in N(v)$, $T-vw$ consists of two sub-trees; $S(T,v,w)$ is the sub-tree containing $w$. Then, 
$$
g(\T,r) = 1+\prod_{w\in N(r)} g(S(\T, r, w), w) 
$$
and $g(T,v)$ can be (recursively) computed as follows: if $(T,v) \in \F = \bigcup_{k \in [K]} \F_k$, then $g(T,v)$ is already known (that is, computed by computer); otherwise,
\begin{equation*}
g(T,v)=1+\prod_{w\in N(v)} g(S(T, v, w), w) = m(T,v) \cdot \prod_{w\in N(v)} g(S(T, v, w), w),
\end{equation*}
where 
$$
m(T,v) = \frac {1+\prod_{w\in N(v)} g(S(T, v, w), w) } {\prod_{w\in N(v)} g(S(T, v, w), w) } = \frac {g(T,v)}{g(T,v)-1}.
$$

Clearly, for any $(T,v) \in \F_k$ we have the following trivial upper bound: $m(T,v) \le (k+1)/k$; this bound is sharp as $g(T,v)=(k+1)/k$ for a rooted path on $k$ vertices. We will use this bound for all pairs $(T,v)$ where $v$ is a leaf in $T$. The number of pairs $(T,v)$ in $\F_k$ where $v$ is a leaf of $T$ is $k (k-1)^{k-2}$ (there are $k$ choices for the label of $v$, and $(k-1)^{k-2}$ rooted trees in $\F_{k-1}$ that can be attached to $v$ to form $T$). This explains the term $\xi_2$. For heavy internal vertices, we use even a weaker bound: $m(T,v) \le (\hat{K}+1)/\hat{K}$. This justifies the term $\xi_1$. To make our bound stronger, we will use a better estimation for $m(T,v)$ when $v$ has degree at least 2 in $T$ and corresponds to a light internal vertex in $\T$. Indeed, if this is the case, then $g(T,v) \ge 2k-1$; this bound is also sharp as $g(T,v) = 2(k-1)+1 = 2k-1$ for a rooted path on $k-1$ vertices with a leaf attached to the root (that is, a path on $k$ vertices rooted at a vertex adjacent to a leaf). Hence, for pairs of this type we have $m(T,v) \le (2k-1)/(2k-2)$. The total number of members of $\F_k$ is $k^{k-1}$ and we already know how many of them are not of this type. This justifies the term $\xi_3$.

Putting all ingredients together we get that a.a.s.\ $g(\T,r) \le \left( \xi_1 \xi_2 \xi_3 \xi_4 + o(1) \right)^n$, and so $c(n) \le n \left( \xi_1 \xi_2 \xi_3 \xi_4 + o(1) \right)^n = \left( \xi_1 \xi_2 \xi_3 \xi_4 + o(1) \right)^n$ as $n = (1+O(\log n / n))^n = (1+o(1))^n$. The proof is finished.
\end{proof}

The numerical values of the upper bounds for $c(n)$ ($K \in \{1, 2, \ldots, 30\}$ and $\hat{K} = 10,000$) following from Theorem~\ref{thm:upper-comp} are presented in Table~\ref{tab:tab1} (see column \emph{upper 2}). Note that in the computations we are aggregating very small numbers; therefore, in order to ensure numerical soundness of the results, we have performed them using 1,024 bit mantissa and rounding-up arithmetic. For $K=30$ we get the following values
$\xi_1<1.0000008$, $\xi_2<1.0005917$, $\xi_3<1.00049672$, $\xi_4<1.4180539$ that lead to the following upper bound which is the strongest bound we managed to obtain: 

\begin{theorem}\label{thm:upper-computer}
A.a.s.\ $c(n) \le 1.41960^n$.
\end{theorem}

\section{Conclusions}\label{sec:conclustion}

We finish the paper with a few comments. 

\subsection{Conjecture}

Let us revisit the proof of Theorem~\ref{thm:upper-comp}. It follows that the ratio between upper and lower bound for $\sqrt[n]{c(n)}$ can be made arbitrarily close to
$$
\eta := \prod_{k \ge K+1} \prod_{(T,v) \in \F_k} m(T,v)^{e^{-k}/k!} = \prod_{k \ge K+1} m(k)^{k^{k-1} e^{-k}/k!},
$$
where 
$$
m(k) := \left( \prod_{(T,v) \in \F_k} m(T,v) \right)^{1/k^{k-1}}
$$
is a geometric mean of $m(T,v)$ over all members of $\F_k$. We partitioned $\F_k$ into two sets to get the two corresponding upper bounds for $m(T,v)$ ($(k+1)/k$ and $(2k-1)/(2k-2)$) which yielded constants $\xi_2$ and $\xi_3$. The improvement after partitioning of $\F_k$ is rather mild and the main reason for that was to determine the two significant digits of $\sqrt[n]{c(n)}$. On the other hand, one can easily partition $\F_k$ into more sets to improve the upper bound. We do not follow this approach as the following, much stronger, property must be true. It is safe to conjecture that $m(k)$ is a decreasing function of $k$ and this is verified to be the case for $1 \le k \le K=30$---see Table~\ref{tab:tab1} (column \emph{multiplier}). (In fact, it should converge to zero quite fast so the conjecture is really safe.) Unfortunately, at present, we cannot prove this property; we have tried a number of couplings between $\F_{k+1}$ and $\F_k$ but with no success. If the property holds, then
$$
\eta = \prod_{k \ge K+1} m(k)^{k^{k-1} e^{-k}/k!} \le \prod_{k \ge K+1} m(K)^{k^{k-1} e^{-k}/k!} = m(K)^{|R|/n}.
$$
Using $K=30$ and the numerical value of $m(30) \approx 1.00003886$ we make the following conjecture. The conjectured bounds implied by smaller values of $K$ can be found in Table~\ref{tab:tab1} (see column \emph{connj.\ upper}). 

\begin{conjecture}\label{con:upper_bound}
A.a.s.\ $c(n) \le 1.41806182^n$.
\end{conjecture}

In fact, it feels safe to conjecture that the first 5 digits of $\sqrt[n]{c(n)}$ are 1.41805. If the desired property is proved, we would certainly go for $k=31$ to squeeze the last drop from the argument.

\subsection{Upper bound based on degree distribution}

Let $\Ss_\pi$ be the family of all trees on $n$ vertices with a given non-increasing degree sequence $\pi = (d_0, d_1, \ldots, d_{n-1})$. As mentioned in the introduction, it is known which extremal tree from $\Ss_\pi$ has the largest number of sub-trees~\cite{Zhang2013}. This tree, $\T_\pi$, can be constructed in a greedy way using the breadth-first search method. First, label the vertex with the largest degree $d_0$ as $v_{01}$. Then, label the neighbours of $v_{01}$ as $v_{11},v_{12},\ldots, v_{1d_0}$ from ``left to right'' and let $d(v_{1i}) = d_i$ for $i = 1,\ldots, d_0$. Then repeat this for all newly labelled vertices until all degrees are assigned.

As computed in~(\ref{eq:degree_distr}), a.a.s.\ the number of vertices of degree $k$ is $(1+o(1)) n e^{-1}/(k-1)!$. Using that and the construction mentioned above, we get that a.a.s.\ $c(n) \le 1.52745^n$ which gives a non-trivial bound but is far away from the one we obtained. Of course, this is not too surprising, and it confirms that the degree distribution is not a crucial factor in our problem; the number of sub-trees of $\T$ is governed by the distribution of small rooted trees from $\F = \bigcup_{k=1}^K \F_k$.

\subsection{Simulations}

In order to check the behaviour of $c(n)$ for small values of $n$ we have used Monte Carlo simulation approach. For values of $n$ up to $n=1,048,576$, we report estimate of $\sqrt[n]{\E(c(n))}$, median of its distribution, and bounds of 90\% confidence interval obtained using bootstrapping (see Table~\ref{tab:tab2}). Observe that, because of using 1,024 independent replicates of the simulation, the obtained estimates are very precise.
Already for $n=1,048,576$ the mean lies within the asymptotic bounds for $c(n)$ that we were able to prove in the paper and has three decimal digits of agreement with the conjectured value of $c(n)$.

\bibliographystyle{apalike}
\bibliography{trees}

\begin{table}
    \centering
    \begin{tabular}{|r||c|c|c|c|c|c|}
        \hline 
        $K$ & \emph{lower} & \emph{upper 1} & \emph{upper 2} & \emph{conj.\ upper} & $|R|/n$ & \emph{multiplier} \\
        \hline 
        1  & 1.29045464 & 2.00000000 & 1.43208050 & 2.00000000 & 0.63212055 & 2.00000000 \\
        2  & 1.36324560 & 1.92362926 & 1.43208050 & 1.66745319 & 0.49678527 & 1.50000000 \\
        3  & 1.39061488 & 1.86325819 & 1.43208050 & 1.55596710 & 0.42210467 & 1.30495588 \\
        4  & 1.40310215 & 1.81740886 & 1.43138632 & 1.50231117 & 0.37326296 & 1.20085291 \\
        5  & 1.40946163 & 1.78177319 & 1.43028338 & 1.47223973 & 0.33816949 & 1.13753267 \\
        6  & 1.41293442 & 1.75332505 & 1.42906778 & 1.45396472 & 0.31139897 & 1.09628284 \\
        7  & 1.41492327 & 1.73007752 & 1.42789078 & 1.44231043 & 0.29011286 & 1.06831339 \\
        8  & 1.41610182 & 1.71070328 & 1.42681559 & 1.43464704 & 0.27266454 & 1.04887463 \\
        9  & 1.41681820 & 1.69428865 & 1.42586149 & 1.42950426 & 0.25802502 & 1.03515096 \\
        10 & 1.41726225 & 1.68018579 & 1.42502733 & 1.42600450 & 0.24551402 & 1.02536358 \\
        11 & 1.41754178 & 1.66792305 & 1.42430332 & 1.42359935 & 0.23466147 & 1.01833777 \\
        12 & 1.41771993 & 1.65714906 & 1.42367665 & 1.42193478 & 0.22513081 & 1.01327326 \\
        13 & 1.41783464 & 1.64759676 & 1.42313429 & 1.42077680 & 0.21667390 & 1.00961308 \\
        14 & 1.41790914 & 1.63905962 & 1.42266416 & 1.41996815 & 0.20910325 & 1.00696374 \\
        15 & 1.41795786 & 1.63137546 & 1.42225550 & 1.41940180 & 0.20227419 & 1.00504449 \\
        16 & 1.41798993 & 1.62441515 & 1.42189906 & 1.41900425 & 0.19607309 & 1.00365361 \\
        17 & 1.41801115 & 1.61807465 & 1.42158697 & 1.41872470 & 0.19040929 & 1.00264559 \\
        18 & 1.41802525 & 1.61226923 & 1.42131256 & 1.41852784 & 0.18520944 & 1.00191512 \\
        19 & 1.41803467 & 1.60692913 & 1.42107030 & 1.41838904 & 0.18041349 & 1.00138592 \\
        20 & 1.41804099 & 1.60199646 & 1.42085548 & 1.41829108 & 0.17597172 & 1.00100264 \\
        21 & 1.41804523 & 1.59742274 & 1.42066420 & 1.41822188 & 0.17184260 & 1.00072515 \\
        22 & 1.41804809 & 1.59316706 & 1.42049318 & 1.41817297 & 0.16799109 & 1.00052431 \\
        23 & 1.41805003 & 1.58919466 & 1.42033966 & 1.41813836 & 0.16438742 & 1.00037899 \\
        24 & 1.41805134 & 1.58547582 & 1.42020132 & 1.41811387 & 0.16100611 & 1.00027389 \\
        25 & 1.41805223 & 1.58198494 & 1.42007621 & 1.41809652 & 0.15782519 & 1.00019789 \\
        26 & 1.41805284 & 1.57869987 & 1.41996266 & 1.41808423 & 0.15482562 & 1.00014295 \\
        27 & 1.41805326 & 1.57560132 & 1.41985928 & 1.41807551 & 0.15199081 & 1.00010326 \\
        28 & 1.41805354 & 1.57267244 & 1.41976483 & 1.41806933 & 0.14930621 & 1.00007457 \\
        29 & 1.41805374 & 1.56989841 & 1.41967830 & 1.41806494 & 0.14675899 & 1.00005383 \\
        30 & 1.41805387 & 1.56726614 & 1.41959880 & 1.41806182 & 0.14433784 & 1.00003886 \\
        \hline
    \end{tabular}
    
    \caption{Asymptotic lower/upper bounds for $c(n)$, $|R|/n$, and multipliers for various values of $K$.\label{tab:tab1}}
\end{table}

\begin{table}
    \centering
    \begin{tabular}{|r|c|c|c|c|c|}
        \hline 
        $n$ \emph{(tree size)} & \emph{mean} & \emph{5\%} & \emph{50\%} & \emph{95\%} \\
        \hline
        4,096     & 1.432278 & 1.429591 & 1.432278 & 1.432662 \\
        16,384    & 1.424623 & 1.424245 & 1.424615 & 1.424704 \\
        65,536    & 1.421405 & 1.420636 & 1.421405 & 1.421429 \\
        262,144   & 1.419931 & 1.419516 & 1.419931 & 1.419937 \\
        1,048,576 & 1.418849 & 1.418774 & 1.418849 & 1.418851 \\
        \hline
    \end{tabular}
    \caption{Estimation of $\sqrt[n]{\E(c(n))}$ for different finite tree sizes $n$ estimated using simulation (1,024 independent replications). 5\%, 50\% and 95\% percentiles of mean estimator distribution obtained using $10^5$ bootstrap replications. \label{tab:tab2}}
\end{table}

\end{document}